\documentclass[reqno,12pt]{amsart}
\usepackage{amsmath,amscd, amssymb,color,hyperref,mathrsfs,mathtools,stmaryrd}
\usepackage[usenames,dvipsnames]{xcolor}
\usepackage[curve,matrix,arrow]{xy}

\numberwithin{equation}{section}
\newtheorem{theorem}[equation]{Theorem}
\newtheorem{lemma}[equation]{Lemma}

\newtheorem{Th}{Theorem}

\theoremstyle{definition}

\newtheorem{remark}[equation]{Remark}

\newtheorem{eg}[equation]{Example}

\renewcommand{\phi}{\varphi}

\newcommand{\F}{\mathcal{F}}
\newcommand{\G}{\mathcal{G}}

\newcommand{\Hom}{\operatorname{Hom}}
\newcommand{\Inj}{\operatorname{Inj}}
\newcommand{\Aut}{\operatorname{Aut}}
\newcommand{\Inn}{\operatorname{Inn}}

\newcommand{\res}{\operatorname{res}}

\newcommand{\m}{\mathcal}

\def\h{\mathfrak{hyp}}

\def\Hom{\mathrm{Hom}}

\def\Aut{\mathrm{Aut}}
\def\Inn{\mathrm{Inn}}

\def\Syl{\mathrm{Syl}}

\def\Id{\mathrm{Id}}
\def \<{\langle }
\def \>{\rangle }

\theoremstyle{remark}

\numberwithin{equation}{section}

\begin{document}
\title{Control of fusion by abelian subgroups of the hyperfocal subgroup}


\author[E.~Henke]{Ellen Henke}
\address{Institute of Mathematics, 
University of Aberdeen, Fraser Noble Building, Aberdeen AB24 3UE, U.K.}
\email{ellen.henke@abdn.ac.uk}

\author[J.~Liao]{Jun Liao}
\address{School of Mathematics and Statistics,
Hubei University,
Wuhan, $430062$, P. R. China.}
\curraddr{Institute of Mathematics, University of Aberdeen, Fraser Noble Building, Aberdeen AB24 3UE, U.K.}
\email{jliao@pku.edu.cn} 
\thanks{The Second author was supported by the National Natural Science Foundation of China (11401186).}


\date{}

\dedicatory{}

\begin{abstract}
We prove that an isomorphism between saturated fusion systems over the same finite $p$-group
is detected on the elementary abelian subgroups of the hyperfocal subgroup if $p$ is odd, and on the abelian subgroups of the hyperfocal subgroup of exponent at most $4$ if $p=2$. For odd $p$, this has implications for mod $p$ group cohomology.
\end{abstract}

\maketitle

\section{Introduction}

In 1971, Quillen \cite{Q} published two articles relating properties of the mod $p$ cohomology ring of a group $G$ to the elementary abelian $p$-subgroups of $G$. The results hold for any prime $p$ and any group $G$ which is a compact Lie group (e.g. a finite group). Quillen studied in particular varieties of mod $p$ cohomology rings and proved a stratification theorem stating that the variety of the mod $p$ cohomology ring of $G$ can be broken up into pieces corresponding to the $G$-conjugacy classes of elementary abelian $p$-subgroups of $G$.\footnote{More precisely, Quillen studied the variety of the commutative subring of $H^*(G;\mathbb{F}_p)$ of elements of even degree. However, his stratification theorem holds similarly for the variety of $H^*(G;\mathbb{F}_p)$; see Remark~\ref{EvenDegreeRemark}.} Therefore, it is of interest to study conjugacy relations between elementary abelian $p$-subgroups. From now on we assume that $G$ is finite and $H$ is a subgroup of $G$ of index prime to $p$. For any two subgroups $A$ and $B$ of $G$, we write $\Hom_G(A,B)$ for the set of group homomorphisms from $A$ to $B$ that are obtained via conjugation by an element of $G$. As a consequence of Quillen's stratification theorem,  $H$ controls fusion of elementary abelian subgroups in $G$, if the inclusion map from $H$ to $G$ induces an isomorphism between the varieties of the mod $p$ cohomology rings of $H$ and $G$. Here we say that the subgroup $H$ controls fusion of elementary abelian subgroups in $G$ if $\Hom_H(A,B)=\Hom_G(A,B)$ for all elementary abelian subgroups $A$ and $B$ of $H$. Similarly we say that $H$ controls $p$-fusion in $G$ if $\Hom_H(A,B)=\Hom_G(A,B)$ for all $p$-subgroups $A$ and $B$ of $H$. By the Cartan--Eilenberg stable elements formula \cite[XII.10.1]{CE}, the inclusion map from $H$ to $G$ induces an isomorphism in mod $p$ group cohomology if $H$ controls fusion in $G$. Together with Quillen's fundamental results, this motivates the study of connections between control of fusion of elementary abelian subgroups and control of $p$-fusion. 

\smallskip

If $H=S$ is a Sylow $p$-subgroup of $G$ and $p$ is odd, Quillen \cite{Q1} proved as a first illustration of his theory that $G$ is nilpotent if the inclusion map from $S$ to $G$ induces an isomorphism between the corresponding varieties. We recall that, by a classical theorem of Frobenius, $G$ is nilpotent if and only if $S$ controls fusion in $G$. So Quillen showed that $S$ controls fusion in $G$ if $S$ controls fusion of elementary abelian subgroups. Variations of this theorem were proved in \cite{Henn,Brun,Gon,IN,Can,BESW}, but all maintaining the hypothesis that $H=S$ is a Sylow $p$-subgroup. Only relatively recently, Benson, Grodal and the first author of this paper proved a result that holds more generally for any subgroup $H$ of index prime to $p$; see \cite{BGH}. More precisely, it is shown that $H$ controls fusion in $G$ (and thus the inclusion map from $H$ to $G$ induces an isomorphism in mod $p$ group cohomology), if the inclusion map induces an isomorphism between the corresponding varieties, i.e. if $H$ controls fusion of elementary abelian subgroups of $G$. This is obtained as a consequence of a theorem that is stated and proved for saturated fusion systems; see \cite[Theorem~B]{BGH}. In this short note, we point out that actually a slightly stronger version of this theorem holds. We refer the reader to \cite[Part~I]{AKO} for an introduction to fusion systems.

\begin{Th}[Small exponent abelian subgroups of the hyperfocal subgroup control fusion] \label{MainThm}
Let $\G \subseteq \F$ be two saturated fusion systems over the same
finite $p$-group $S$. 
Suppose that
$\Hom_\G(A,B)=\Hom_\F(A,B)$  for all $A,B \leq \h(\F)$ with $A,B$
elementary abelian if $p$ is odd, and abelian of exponent at
most $4$ if $p=2$. Then $\G=\F$.
\end{Th}

If one replaces $\h(\F)$ by $S$, then the above theorem coincides with \cite[Theorem~B]{BGH}. We recall  that the hyperfocal subgroup $\h(\F)$ is the subgroup of $S$ generated by all elements of the form $x^{-1}\phi(x)$ where $x\in Q$ and $\phi\in O^p(\Aut_\F(Q))$ for some subgroup $Q$ of $S$. If $\F=\F_S(G)$ is the fusion system of a finite group $G$ with Sylow $p$-subgroup $S$, then Puig's hyperfocal subgroup theorem \cite[\S 1.1]{P5} states that $\h(\F)=O^p(G)\cap S$. In the situation of Theorem~\ref{MainThm}, Quillen's example $Q_8\leq Q_8:C_3$ shows that it is indeed not enough to consider only elementary abelian subgroups for $p=2$.

\smallskip

A fusion system $\F$ on $S$ is called nilpotent if $\F=\F_S(S)$. Restricting attention to subgroups of the hyperfocal subgroup is motivated by a theorem of the second author of this paper together with Zhang, which characterizes $p$-nilpotency of a saturated fusion system $\F$ by the fusion on certain subgroups of the hyperfocal subgroup of $\F$; see \cite{LZ}. Another motivation comes from work of Ballester-Bolinches, Ezquerro, Su and Wang \cite{BESW} showing that, in certain special cases, fusion is detected on the subgroups of the focal subgroup of $\F$ which are cyclic of order $p$ or $4$. We show here that in Theorem A and C of \cite{BESW}, the focal subgroup can actually be replaced by the hyperfocal subgroup. More precisely, we prove the following theorem which gives in particular a new characterization of nilpotent fusion systems:

\begin{Th}\label{ThmBESW}
Let $\F$ be a saturated fusion system over a
finite $p$-group $S$, and let $\G=N_\F(S)$ or $\G=\F_S(S)$.  
Suppose that $\Hom_\G(A,B)=\Hom_\F(A,B)$  for all $A,B \leq \h(\F)$ which are 
cyclic subgroups of order $p$ or $4$. Then $\G=\F$. 
\end{Th}

We remark that, in general, it is not the case that the subgroup $H$ controls fusion in $G$ if it controls fusion on cyclic subgroups of order $p$ for odd $p$, or on subgroups of order at most $4$ for $p=2$. This is not even the case if $G$ has a normal Sylow $p$-subgroup as the following example shows: Let $n$ be an integer such that $n\geq 2$ and $p$ does not divide $n$. Let $S$ be the field of order $p^n$, so that $S$ under addition forms in particular an elementary abelian group of order $p^n$. Note that every non-zero element of $S$ induces a group automorphism of $S$ via multiplication. Let $D$ be the group of all these automorphisms. Then $D$ is a subgroup of $\Aut(S)$ of order $p^n-1$ acting freely and transitively on the non-trivial elements of $S$. Let $\sigma$ be the Frobenius automorphism of the field $S$. Then $\sigma$ has order $n$ and is also a group automorphism of $S$. Moreover, $\sigma$ normalizes $D$, as conjugation by $\sigma$ takes every element of $D$ to its $p$th power. Hence, $\hat{D}=D\rtimes\<\sigma\>$ is a group of order $(p^n-1)n$. Since $p$ does not divide $n$, it follows that $S$ is a normal Sylow $p$-subgroup of $G:=S\rtimes \hat{D}$. Moreover, $H:=S\rtimes D$ is a subgroup of $G$ of index prime to $p$. Note also that $S=[S,D]=\h(\F_S(G))$. Let $\m{V}$ be the set of subgroups of $S$ of order $p$. Then $\m{V}$ has $\frac{p^n-1}{p-1}$ elements. As $D$ acts freely and transitively on the non-trivial elements of $S$, it follows that $D$ acts also transitively on $\m{V}$, and that $C_D(A)=1$ for all $A\in\m{V}$. Thus $|\Aut_D(A)|=|N_D(A)|=\frac{|D|}{|\m{V}|}=p-1$ for every $A\in\m{V}$. As any two elements of $\m{V}$ are conjugate under $D$, it follows $|\Hom_D(A,B)|=p-1$ for all $A,B\in\m{V}$. Thus, $\Hom_H(A,B)=\Hom_D(A,B)$ is the set $\Inj(A,B)$ of injective group homomorphism from $A$ to $B$. As $\Hom_H(A,B)\subseteq\Hom_G(A,B)\subseteq\Inj(A,B)$, this implies $\Hom_H(A,B)=\Hom_G(A,B)$ for all $A,B\in\m{V}$.  
So $H$ controls fusion in $G$ of the cyclic subgroups of order $p$ (and thus for $p=2$ also of the cyclic subgroups of order at most $4$). However, as $D\neq \hat{D}$, the subgroup $H$ does not control fusion in $G$. 

\smallskip

We conclude by stating a version of Theorem~\ref{MainThm} in terms of varieties of cohomology rings. We continue to assume that $G$ is a finite group and we fix moreover an algebraically closed field $\Omega$ of prime characteristic $p$. We either set $k=\Omega$ or $k=\mathbb{F}_p$. Moreover, set  $H^\ast(G):=H^\ast(G,k)$ and define the variety $V_G$ to be the variety $\Hom_k(H^\ast(G),\Omega)$ of $k$-algebra homomorphisms from $H^\ast(G)$ to $\Omega$; see Remark~\ref{EvenDegreeRemark} for alternative definitions of $V_G$. Then every $k$-algebra homomorphism $\alpha\colon H^\ast(G)\rightarrow H^\ast(H)$ induces a map of varieties $\alpha^*\colon V_H\rightarrow V_G$ by sending any homomorphism $\beta\in V_H=\Hom_k(H^\ast(H),\Omega)$ to $\beta\circ\alpha\in V_G=\Hom_k(H^\ast(G),\Omega)$. For an arbitrary subgroup $H$ of $G$, we write 
$\res_{G,H}: H^{\ast}(G)\longrightarrow H^{\ast}(H)$
for the map induced by the inclusion map $H\rightarrow G$, and hence  
$\res^{\ast}_{G,H}: V_H\longrightarrow V_G$ for the corresponding map of varieties. 

\smallskip

If $H$ is a subgroup of $G$ containing a Sylow $p$-subgroup $S$ of $G$, then we have the inclusion maps $S\cap O^p(G)\hookrightarrow H\hookrightarrow G$ which induce the following maps of varieties:
\[ V_{S\cap O^p(G)}\xrightarrow{\;\res_{H,S\cap O^p(G)\;}^\ast} V_H\xrightarrow{\;\;\;\;\res_{G,H}^\ast\;\;\;\;}V_G\]
So in particular, we can consider the restriction of the map $\res_{G,H}^\ast\colon V_H\rightarrow V_G$ to the subvariety $\res_{H,S\cap O^p(G)}^{\ast}V_{S\cap O^p(G)}$ of $V_H$. If $p$ is an odd prime and $H$ is a subgroup of $G$ of index prime to $p$, then the results in \cite{BGH} say basically that $H$ controls fusion in $G$ if $\res_{G,H}^\ast\colon V_H\rightarrow V_G$ is an isomorphism of varieties. Theorem~\ref{MainThm} implies a slightly stronger statement which is stated in the next theorem. Notice that a subgroup $H$ of $G$ has index prime to $p$ if and only if $H$ contains a Sylow $p$-subgroup of $G$.

\begin{Th}\label{App}
Let $G$ be a finite group, let $p$ be an odd prime, and let  $H$ be a subgroup of $G$ containing a Sylow $p$-subgroup $S$ of $G$. Suppose the restriction of the map $\res_{G,H}^{\ast}$ to $\res^{\ast}_{H,S\cap O^{p}(G)}V_{S\cap O^{p}(G)}$ is injective. Then $H$ controls fusion in $G$ and the restriction map $\res_{G,H}: H^{\ast}(G)\longrightarrow H^{\ast}(H)$ is an isomorphism.
\end{Th}

Note that Theorem~\ref{App} says in particular that the map $\res_{G,H}^\ast\colon V_H\rightarrow V_G$ is an isomorphism of varieties if its restriction to $\res^{\ast}_{H,S\cap O^{p}(G)}V_{S\cap O^{p}(G)}$ is injective. One sees easily that the converse of Theorem \ref{App} holds as well: If  $\res_{G,H}: H^{\ast}(G)\longrightarrow H^{\ast}(H)$ is an isomorphism then $\res_{G,H}^{\ast}: V_{H}\longrightarrow V_{G}$ is an isomorphism. In particular, the restriction of $\res_{G,H}^\ast$ to $\res_{H,S\cap O^p(G)}V_{S\cap O^p(G)}$ is injective.

\smallskip

We remark also that a theorem analogous to Theorem~\ref{App} can be proved for saturated fusion systems rather than for groups. For more details, we refer the reader to Remark~\ref{FusionSystemCohomology}.

\subsection*{Acknowledgement} The authors would like to thank Dave Benson for his patient explanations regarding group cohomology.

\section{Proof of Theorem~\ref{MainThm} and Theorem~\ref{ThmBESW}}

The proof of Theorem~\ref{MainThm} is very similar to the proof of Theorem~B in \cite{BGH}. We need the following variation of \cite[Theorem~2.1]{BGH}.

\begin{theorem}\label{ThompsonThm}
 Let $P$ be a finite $p$-group and let $G$ be a subgroup of $\Aut(P)$ containing the group $\Inn(P)$ of inner automorphism. Then there exists a $G$-invariant subgroup $D$ of $[P,O^p(G)]$, of exponent $p$ if $p$ is odd and exponent at most $4$ if $p=2$, such that $[D,P]\leq Z(D)$, and such that every non-trivial $p^\prime$-automorphism in $G$ restricts to a non-trivial $p^\prime$-automorphism of $D$. Furthermore, for any such $D$ and any maximal (with respect to inclusion) abelian subgroup $A$ of $D$ it follows that $A\unlhd P$ and $C_G(A)$ is a $p$-group.
\end{theorem}

\begin{proof}
 By \cite[Theorem~2.1]{BGH}, there exists a characteristic subgroup $D_1$ of $P$, of exponent $p$ if $p$ is odd and exponent at most $4$ if $p=2$, such that $[D_1,P]\leq Z(D_1)$, and such that every non-trivial $p^\prime$-automorphism of $P$ restricts to a non-trivial $p^\prime$-automorphism of $D_1$. Set $D:=[D_1,O^p(G)]$. As $D_1$ is $G$-invariant and as $O^p(G)$ is normal in $G$, the subgroup $D$ is $G$-invariant. In particular, as $\Inn(P)\leq G$ by assumption, we have $[D,P]\leq D$. Using $[D_1,P]\leq Z(D_1)$ we obtain thus $[D,P]\leq [D_1,P]\cap D\leq Z(D_1)\cap D\leq Z(D)$. If $\phi$ is a $p^\prime$-automorphism of $P$ with $\phi|_D=\Id_D$ then $[D,\phi]=1$ and $[D_1,\phi]\leq [D_1,O^p(G)]=D$. Thus, by \cite[Theorem~5.3.6]{G}, we have $[D_1,\phi]=[D_1,\phi,\phi]\leq [D,\phi]=1$ and $\phi|_{D_1}=\Id_{D_1}$. Because of the way $D_1$ was chosen, this implies that $\phi=\Id_P$. So we have shown that every non-trivial $p^\prime$-automorphism in $G$ restricts to a non-trivial automorphism of $D$. 

\smallskip

For the last part let $A$ be a maximal subgroup of $D$ with respect to inclusion. Then $[A,P]\leq Z(D)\leq A$ and thus $A\unlhd P$. Furthermore, if $B\leq C_G(A)$ is a $p^\prime$-subgroup, then $A\times B$ acts on $D$. Since $A$ is maximal abelian, it follows $C_D(A)=A\leq C_D(B)$. Thompson's $A\times B$-lemma \cite[Theorem~5.3.4]{G} now says that $[D,B]=1$ and so $B=1$. Since $B$ was arbitrary, it follows that $C_G(A)$ is a $p$-group.   
\end{proof}

We need the following crucial lemma, which is \cite[Main Lemma~2.4]{BGH}. 

\begin{lemma}\label{MainLemma}
 Let $\G\subseteq\F$ be two saturated fusion systems on the same finite $p$-group $S$, and $P\leq S$ an $\F$-centric and fully $\F$-normalized subgroup, with $\Aut_\F(R)=\Aut_\G(R)$ for every $P< R\leq N_S(P)$. Suppose that there exists a subgroup $Q\unlhd P$ with $\Hom_\F(Q,S)=\Hom_\G(Q,S)$. Then $\Aut_\F(P)=\<\Aut_\G(P),C_{\Aut_\F(P)}(Q)\>$.
\end{lemma}

\begin{proof}[Proof of Theorem~\ref{MainThm}]
By Alperin's fusion theorem \cite[Theorem~I.3.6]{AKO}, $\F$ is generated by $\F$-automorphisms of fully $\F$-normalized $\F$-centric subgroups. We want to show that $\Aut_\F(P)=\Aut_\G(P)$ for all $P\leq S$. By induction on $|S:P|$, we can assume that $\Aut_\F(R)=\Aut_\G(R)$ for all $R\leq S$ with $|R|>|P|$. Furthermore, by Alperin's fusion theorem, we can choose $P$ to be fully $\F$-normalized and $\F$-centric. By Theorem~\ref{ThompsonThm}, we can pick an $\Aut_\F(P)$-invariant subgroup $D$ of $[P,O^p(\Aut_\F(P))]$, of exponent $p$ if $p$ is odd and of exponent at most $4$ if $p=2$, such that every non-trivial $p^\prime$-automorphism $\phi\in\Aut_\F(P)$ restricts to a non-trivial automorphism of $D$ and, for any maximal (with respect to inclusion) abelian subgroup $A$ of $D$, $A\unlhd P$ and $C_{\Aut_\F(P)}(A)$ is a $p$-group. As $P$ is fully $\F$-normalized, $\Aut_S(P)$ is a Sylow $p$-subgroup of $\Aut_\F(P)$, and so if we replace $A$ by a conjugate of $A$ under $\Aut_\F(P)$, we can arrange that $C_{\Aut_\F(P)}(A)\leq \Aut_S(P)\leq \Aut_\G(P)$. As $D$ has exponent $p$ if $p$ is odd and exponent at most $4$ if $p=2$, we have by assumption $\Hom_\F(A,S)=\Hom_\G(A,S)$. So by Lemma~\ref{MainLemma} applied with $A$ in place of $Q$, we obtain that $\Aut_\F(P)=\<\Aut_\G(P),C_{\Aut_\F(P)}(A)\>=\Aut_\G(P)$ as wanted.
\end{proof}

Let $\mathcal{P}$ be a set of representatives  of the $\F$-conjugacy classes of $\F$-essential subgroups. A version of the Alperin--Goldschmidt Theorem for fusion systems states that $\F$ is generated by the $\F$-automorphism groups of the elements of $\mathcal{P}\cup\{S\}$. Analyzing what is used in the proof above, one sees that we only need the following condition in Theorem~\ref{MainThm}:   For every $P\in\m{P}\cup\{S\}$ and every abelian subgroup $A$ of the commutator subgroup $[P,O^p(\Aut_\F(P))]$ which is of exponent $p$ or $4$, we have $\Hom_\F(A,S)=\Hom_\G(A,S)$. 

\smallskip

The proof of Theorem~\ref{ThmBESW} is essentially the same as the one of \cite[Theorem A]{BESW}
except that we use Theorem~\ref{ThompsonThm} instead of \cite[Theorem~2.1]{BGH}. Essentially, Theorem~\ref{ThmBESW} is a consequence of the following lemma:

\begin{lemma}\label{LemmaBESW}
Let $\F$ be a saturated fusion systems over a finite $p$-group $S$. Suppose that 
$\Hom_\F(A,B)\subseteq \Hom_{N_\F(S)}(A,B)$  for all subgroups $A,B \leq \h(\F)$ which are
cyclic of order $p$ or $4$. Then $\F=N_\F(S)$.
\end{lemma}

\begin{proof}
Suppose that $Q$ is an $\F$-essential subgroup. Then by definition, $Q$ is in particular fully normalized and thus $\Aut_S(Q)$ is a Sylow $p$-subgroup of $\Aut_\F(Q)$. By Theorem \ref{ThompsonThm}, 
there is an $\Aut_{\F}(Q)$-invariant subgroup $D\leq [Q,O^p(\Aut_\F(Q))]\leq Q\cap\h(\F)$ such that every non-trivial $p^\prime$-element of $\Aut_{\F}(Q)$ restricts to a non-trivial automorphism of $D$, and $D$ is of exponent $p$ or $4$.
Let $Z_{i}(S)$ be the $i$-th center of $S$ and $D_{i}=D\cap Z_{i}(S)$.
We argue now that $D_{i}$ is $\Aut_{\F}(Q)$-invariant: For every $x\in D_i$ and any $\phi\in \Aut_\F(Q)$, $\phi|_{\<x\>}$ extends by hypothesis to an element of $\Aut_\F(S)$ which clearly normalizes $Z_i(S)$. As $\phi$ normalizes $D$, it follows $\phi(x)\in Z_i(S)\cap D=D_i$. So $D_i$ is indeed $\Aut_\F(Q)$-invariant. Thus, for some $n\in\mathbb{N}$, the series $1=D_{0}\leq D_{1}\leq \cdots\leq D_{n}=D$ is $\Aut_{\F}(Q)$-invariant. So the stabilizer $H$ of this series (i.e. the set of elements in $\Aut_{\F}(Q)$ acting trivially on $D_i/D_{i-1}$ for each $i\leq n$) forms a normal subgroup of $\Aut_\F(Q)$. For any $p^\prime$-element $\varphi$ of $H$, we have $\varphi|_{D}=\Id_{D}$ by \cite[Theorem 5.3.2]{G}, and thus $\varphi=\Id_{Q}$ by the choice of $D$. Therefore, the stabilizer $H$ is a $p$-group and so $H\leq O_{p}(\Aut_{\F}(Q))$. Since  $\Aut_{S}(Q)$ stabilizes the series $D_0\leq D_1\leq\cdots\leq D_n=D$, it follows that  $\Aut_{S}(Q)=O_{p}(\Aut_{\F}(Q))$, which is a contradiction as every $\F$-essential subgroup is centric and radical. Hence there is no $\F$-essential subgroup. Thus, $\F=N_{\F}(S)$ by Alperin's fusion theorem \cite[Theorem~I.3.6]{AKO}.
\end{proof}

\begin{proof}[Proof of Theorem~\ref{ThmBESW}]
Lemma~\ref{LemmaBESW}, $\F=N_\F(S)$. So for $\G=N_\F(S)$ the assertion follows immediately. Assume now $\G=\F_{S}(S)$.  As $\F=N_\F(S)$, it is sufficient to show that $\Aut_\F(S)=\Inn(S)$. By Theorem \ref{ThompsonThm}, 
there is an $\Aut_{\F}(S)$-invariant subgroup $D\leq [S,O^p(\Aut_\F(S))]\leq S\cap\h(\F)$ such that every non-trivial $p^\prime$-element of $\Aut_{\F}(S)$ restricts to a non-trivial automorphism of $D$, and $D$ is of exponent $p$ or $4$. Let $D_{i}=D\cap Z_{i}(S)$ and $n\in\mathbb{N}$ such that $D_n=D$. By hypothesis, every element of $\Aut_\F(S)$ acts on every element of $D$ as conjugation by an element of $S$. Hence, $\Aut_{\F}(S)$ stabilizes the series $1=D_0\leq D_1\leq\dots\leq D_n=D$ and is thus a $p$-group by \cite[Theorem~5.3.2]{G}. Since $\Inn(S)\in\Syl_{p}(\Aut_{\F}(S))$,
it follows $\Aut_{\F}(S)=\Inn(S)$ as required.
\end{proof}

\section{Proof of Theorem~\ref{App}}

Throughout, assume that $G$ is a finite group and that $\Omega$ is an algebraically closed field of prime characteristic $p$. Let $H^{\ast}(G)$ and $V_G$ be as in the introduction. Recall that, for any subgroup $H$ of $G$, we write $\res_{G,H}\colon H^{\ast}(G)\rightarrow H^{\ast}(H)$ for the map induced by the inclusion map from $H$ to $G$, and  $\res_{G,H}^{\ast}\colon V_H\rightarrow V_G$ for the corresponding map of varieties.

\smallskip

For the proof of Theorem~\ref{App} we will need some more notation: For every elementary abelian $p$-group $A$, we set
\[V_A^+:=V_A\backslash \bigcup_{A'<A}\res_{A,A'}^{\ast}V_{A'}.\]
If $A$ is an elementary abelian subgroup of $G$, set 
\[V_{G,A}^+=\res_{G,A}^{\ast}V_A^+.\] 
We start with the following elementary observation:

\begin{remark}\label{Remark}
 Let $A\leq K\leq G$ such that $A$ is elementary abelian. Then  
$\res_{G,K}^*V_{K,A}^+=V_{G,A}^+$. 
\end{remark}

\begin{proof}
 As $\res_{G,K}^{\ast}\circ\res_{K,A}^{\ast}=\res_{G,A}^{\ast}$, we have 
 $V_{G,A}^+=\res_{G,A}^{\ast}V_A^+=\res_{G,K}^{\ast}(\res_{K,A}^{\ast}V_A^+)=\res_{G,K}^{\ast}V_{K,A}^+$.
\end{proof}

\begin{remark}\label{EvenDegreeRemark}
Write $H^{ev}(G)$ for the subring of $H^\ast(G)$ of elements of even degree. If $k=\mathbb{F}_p$ notice that the $k$-algebra homomorphisms from $H^\ast(G)$ to $\Omega$ are the same as the ring homomorphisms from $H^\ast(G)$ to $\Omega$. So if $k=\mathbb{F}_p$ then, upon replacing  $H^\ast(G)$ by $H^{ev}(G)$ if $p$ is odd, the variety $V_G$  corresponds to the variety $H_G(X)(\Omega)$ studied by Quillen \cite{Q} in the special case that $X$ is a point. If $k=\Omega$, it follows from Hilbert's Nullstellensatz that $V_G$ is homeomorphic to the maximal ideal spectrum of $H^\ast(G)$ via the map sending every homomorphism in $V_G$ to its kernel; see Theorem~5.4.2 and the surrounding discussion in \cite{B}. So again upon replacing $H^\ast(G)$ by $H^{ev}(G)$, the variety $V_G$ as defined in this paper corresponds to the variety $V_G$ as defined by Benson \cite{B}. 

\smallskip

It is common to study the variety of $H^{ev}(G)$ rather than the variety of $H^\ast(G)$, because $H^{ev}(G)$ is commutative, whereas $H^\ast(G)$ is only graded commutative, and texts on commutative algebra are written for strictly commutative rings. As pointed out by Benson \cite[p.9]{BComm}, the results from commutative algebra which are needed in the theory hold accordingly for graded commutative rings. Moreover, it is pointed out that any graded commutative ring $A$ is commutative modulo its nilradical, and every element of odd degree lies in the nilradical if $p$ is odd. So writing $\mathfrak{Nil}$ for the nilradical of $H^\ast(G)$, it follows that $H^\ast(G)/\mathfrak{Nil}$ is isomorphic to $H^{ev}(G)/(H^{ev}(G)\cap\mathfrak{Nil})$. As the nilradical $\mathfrak{Nil}$ is contained in the kernel of every $k$-algebra homomorphism from $H^*(G)$ to $\Omega$, the variety $\Hom_k(H^\ast(G),\Omega)$ is canonically homeomorphic to the variety $\Hom_k(H^{ev}(G),\Omega)$. 
\end{remark}

In particular, the Quillen Stratification Theorem as stated in \cite[Theorem~10.2]{Q} and \cite[Theorem~5.6.3]{B} can be proved accordingly with our definitions:

\begin{theorem}[Quillen's Stratification Theorem]\label{QuillenStrat}
 Let $\m{A}$ be a set of representatives of the $G$-conjugacy classes of elementary abelian subgroups of $G$. Then $V_G$ is the disjoint union
\[V_G=\coprod_{A\in \m{A}}V_{G,A}^+.\]
of locally closed subvarieties $V_{G,A}^+$. Moreover, for every $A\in \m{A}$, the automorphism group $\Aut_G(A)$ acts freely on $V_A^+$ and the map $\res_{G,A}^{\ast}$ induces a homeomorphism $V_A^+/\Aut_G(A)\rightarrow V_{G,A}^+$.
\end{theorem}

The fact that $V_G=\coprod_{A\in\m{A}}V_{A,G}^+$ for any set $\m{A}$ of representatives of the $G$-conjugacy classes of the elementary abelian subgroups of $G$, will be used in our proof in the following form: 

\begin{remark}\label{WeakQuillenStrat}
Let $A$ and $A'$ be elementary abelian subgroups of $G$. If $A$ and $A'$ are $G$-conjugate then we have  $V_{G,A}^+=V_{G,A'}^+$, and if $A$ and $A'$ are not $G$-conjugate then $V_{G,A}^+$ and $V_{G,A'}^+$ are disjoint. 
\qed
\end{remark}

\begin{proof}[Proof of Theorem~\ref{App}.]
Assume that the restriction of the map $\res_{G,H}^\ast\colon V_H\rightarrow V_G$ to the subvariety $\res^{\ast}_{H,S\cap O^{p}(G)}V_{S\cap O^{p}(G)}$ of $V_H$ is injective. 

\smallskip

\emph{Step~1:} Let $A$ be an elementary abelian subgroup of $S\cap O^p(G)$. We show that the map $\res_{G,H}^\ast$ induces a bijection from $V_{H,A}^+$ to $V_{G,A}^+$. Moreover, if $A'$ is another elementary abelian subgroup of $S\cap O^p(G)$ such that $V_{G,A}^+=V_{G,A'}^+$, then we show $V_{H,A}^+=V_{H,A'}^+$.

\smallskip 

To see this note that, by Remark~\ref{Remark}, we have that $\res_{G,H}^\ast V_{H,A}^+=V_{G,A}^+$ and that  $V_{H,A}^+=\res_{H,S\cap O^p(G)}^\ast V_{S\cap O^p(G),A}^+$ is contained in $\res^{\ast}_{H,S\cap O^{p}(G)}V_{S\cap O^{p}(G)}$. By a symmetric argument, it follows that $V_{H,A'}^+$ is contained in $\res^{\ast}_{H,S\cap O^p(G)}V_{S\cap O^p(G)}$ and $\res_{G,H}^\ast V_{H,A}^+=V_{G,A}^+=V_{G,A'}^+=\res_{G,H}^\ast V_{H,A'}^+$. As we assume that the restriction of $\res_{G,H}^\ast$ to $\res^{\ast}_{H,S\cap O^p(G)}V_{S\cap O^p(G)}$ is injective, the above assertion follows.

\smallskip

\noindent\emph{Step~2:} Let $A$ and $A'$ be two $G$-conjugate elementary abelian subgroups of $S\cap O^p(G)$. We show that $A$ and $A'$ are $H$-conjugate. For the proof note that $V_{G,A}^+=V_{G,A'}^+$ by Lemma~\ref{WeakQuillenStrat}. So by Step~1, we have $V_{H,A}^+=V_{H,A'}^+$. Thus, again by Lemma~\ref{WeakQuillenStrat} now used with $H$ in place of $G$, the subgroups $A$ and $A'$ need to be $H$-conjugate. This completes the proof of Step~2.

\smallskip

\noindent\emph{Step~3:} Let $A$ be an elementary abelian subgroup of $S\cap O^p(G)$. We show that $\Aut_G(A)=\Aut_H(A)$. By the Quillen stratification theorem Theorem~\ref{QuillenStrat}, the group $\Aut_G(A)$ acts freely on $V_A^+$, and the map  $\res_{G,A}^{\ast}$ induces a homeomorphism $V_A^+/\Aut_G(A)\rightarrow V_{G,A}^+$. In particular, the fibres of the map $\res_{G,A}^{\ast}\colon V_A^+\rightarrow V_{G,A}^+$ are precisely the orbits of $\Aut_G(A)$ on $V_A^+$. Similarly, applying the Quillen stratification theorem with $H$ in place of $G$, we get that $\Aut_H(A)$ acts freely on $V_A^+$, and the fibres of the map $\res_{H,A}^{\ast}\colon V_A^+\rightarrow V_{H,A}^+$ are precisely the orbits of $\Aut_H(A)$ on $V_A^+$. Note that $\res_{G,A}^{\ast}=\res_{G,H}^{\ast}\circ\res_{H,A}^{\ast}$. As the map $\res_{G,H}^\ast\colon V_{H,A}^+\rightarrow V_{G,A}^+$ is by Step~1 a bijection, it follows that the maps $\res_{G,A}^{\ast}\colon V_A^+\rightarrow V_{G,A}^+$ and $\res_{H,A}^{\ast}\colon V_A^+\rightarrow V_{H,A}^+$ have the same fibres. So the $\Aut_G(A)$-orbits on $V_A^+$ are the same as the $\Aut_H(A)$-orbits. As the actions of $\Aut_G(A)$ and $\Aut_H(A)$ on $V_A^+$ are free, this implies that $|\Aut_G(A)|=|\Aut_H(A)|$. Thus, since $\Aut_H(A)\subseteq\Aut_G(A)$, it follows $\Aut_G(A)=\Aut_H(A)$. 

\smallskip

\noindent\emph{Step~4:} We are now in a position to complete the proof. Let $A$ and $A'$ be elementary abelian subgroups of $S\cap O^p(G)$. We want to show that $\Hom_G(A,A')=\Hom_H(A,A')$ and can assume without loss of generality that $A$ and $A'$ are $G$-conjugate. Then $A$ and $A'$ are $H$-conjugate by Step~1 and thus there exists $\psi\in\Hom_H(A,A')$. Let $\phi\in\Hom_G(A,A')$. Note that $\phi=\psi\circ (\psi^{-1}\circ \phi)$ and $\psi^{-1}\circ\phi\in\Aut_G(A)=\Aut_H(A)$ by Step~2. So it follows that $\phi\in\Hom_H(A,A')$ which proves $\Hom_G(A,A')=\Hom_H(A,A')$. By Puig's hyperfocal subgroup theorem \cite[\S 1.1]{P5}, we have $S\cap O^p(G)=\h(\F_S(G))$. So using Theorem~\ref{MainThm}, we can conclude that $\F_S(G)=\F_S(H)$. Thus, by the Cartan--Eilenberg stable elements formula \cite[XII.10.1]{CE}, the map $\res_{G,H}\colon H^{\ast}(G)\rightarrow H^{\ast}(H)$ is an isomorphism.
\end{proof}

\begin{remark}\label{FusionSystemCohomology}
A version of Theorem~\ref{App} can also be formulated and proved for abstract saturated fusion systems rather than for groups. Let $\F$ be a saturated fusion system over a finite $p$-group $S$. Assume that $k$ is an algebraically closed field of characteristic $p$. The cohomology ring $H^\ast(\F)=H^{\ast}(\F,k)$ of the saturated fusion system
$\F$ is the subring of $\F$-stable element in $H^\ast(S)=H^{\ast}(S,k)$,
which is the subring of $H^{\ast}(S)$ consisting of elements $\xi\in H^{\ast}(S)$ such that
$\res^{S}_{P}(\xi)=\res_{\varphi}(\xi)$
for any $\varphi\in \Hom_{\F}(P,S)$ and any subgroup $P\leq S$.
The ring $H^{\ast}(\F)$ is a graded commutative ring. We write $V_{\F}$ for the maximal ideal spectrum of $H^{\ast}(\F)$, or alternatively for the variety of $k$-algebra homomorphisms from $H^\ast(\F)$ to $k$.

Let $\m{G}$ be a saturated fusion subsystem of $\F$. Note that any $\F$-stable element of $H^{\ast}(S)$ is in particular $\m{G}$-stable, so we can consider the inclusion map $\res_{\F,\m{G}}\colon H^\ast(\F)\rightarrow H^\ast(\m{G})$ which then gives us a map $\res_{\F,\m{G}}^\ast\colon V_{\m{G}}\rightarrow V_\F$ of varieties. Similarly, if $Q\leq S$, we are given a $k$-algebra homomorphism $\res_{\F,Q}\colon H^\ast(\F)\rightarrow H^\ast(Q)$ by composing the inclusion map $H^\ast(\F)\hookrightarrow H^\ast(S)$ with the restriction map $\res_{S,Q}\colon H^\ast(S)\rightarrow H^\ast(Q)$. Again, this induces a map of varieties $\res_{\F,Q}^\ast\colon V_Q\rightarrow V_\F$. 
In particular, if $A\leq S$ is elementary abelian, one can define $V_{\F,A}^+=\res_{\F,A}^\ast V_A^+$.   In an unpublished preprint, Markus Linckelmann \cite[Theorem~1]{L} proves a version of the Quillen stratification theorem; see also Theorem~1.3 and Remark~1.1 in \cite{T}. 
Using this, one can similarly prove the following version of Theorem~\ref{App} for fusion systems:

\smallskip

Let $\G\subseteq \F$ be an inclusion of saturated fusion systems over the same finite $p$-group $S$, and $p$ an odd prime. If the restriction of the map $\res_{\F,\G}^\ast\colon V_\G\rightarrow V_\F$ to $\res_{\G,\h(\F)}^\ast V_{\h(\F)}$ is injective, then $\F=\G$ and in particular $H^\ast(\F)=H^\ast(\G)$. 
\end{remark}

\bibliographystyle{amsalpha}

\end{document}